\documentclass[12pt,reqno]{amsart}
\usepackage{amsmath,amssymb,bbm,esint,times,url}

\usepackage{tikz}
\usetikzlibrary{calc,arrows,decorations.text}
\usepackage[bookmarks=false,unicode,colorlinks,urlcolor=red,citecolor=red,linkcolor=blue]{hyperref}
\usepackage{aliascnt}
\usepackage{listings}

\usepackage[margin=1in]{geometry}

\newtheorem*{remark}{Remark}

\newtheorem{theorem}{Theorem}[section]

\newaliascnt{lemma}{theorem}
\newtheorem{lemma}[lemma]{Lemma}
\aliascntresetthe{lemma}

\newaliascnt{proposition}{theorem}

\aliascntresetthe{proposition}

\newaliascnt{corollary}{theorem}
\newtheorem{corollary}[corollary]{Corollary}
\aliascntresetthe{corollary}

\newaliascnt{conjecture}{theorem}
\newtheorem{conjecture}[conjecture]{Conjecture}
\aliascntresetthe{conjecture}

\tikzset{>=stealth'}

\begin{document}

\title[]{On the hot spots conjecture for acute triangles}
\begin{abstract}
  We show that the hot spots conjecture of J. Rauch holds for acute triangles if one of the angles is not larger than $\pi/6$. More precisely, we show that the second Neumann eigenfunction on those acute triangles has no maximum or minimum inside the domain. We first simplify the problem by showing that absence of critical points on two sides implies no critical points inside a triangle. This result applies to any acute triangle and might help prove the conjecture for arbitrary acute triangles. Then we show that there are no critical points on two sides assuming one small angle. We also establish simplicity for the second Neumann eigenvalue for all non-equilateral triangles.
\end{abstract}

\author[]{Bart{\l}omiej Siudeja}
\address{Department of Mathematics, Univ. of Oregon, Eugene, OR 97403, U.S.A.}
\email{Siudeja\@@uouregon.edu}
\date{\today}


\maketitle

\section{Introduction}

The hot spots conjecture, posed by J. Rauch, states that the hottest point on an insulated plate with ``almost arbitrary'' initial heat distribution moves (as time passes) toward the boundary of the plate. This physical phenomenon can be mathematically formalized using a Laplace eigenfunction problem. More precisely, consider the eigenvalue problem with Neumann (natural) boundary condition
\begin{align*}
  \Delta \varphi&=-\mu \varphi,\text{ on } D\\
  \frac{\partial \varphi}{\partial n}&=0,\text{ on }\partial D.
\end{align*}
It is well known that for nice enough bounded domains $D$ (in particular convex, or with smooth boundary) there exists an increasing sequence of eigenvalues satisfying
\begin{align*}
  0=\mu_1<\mu_2\le \mu_3\le \dots \to \infty.
\end{align*}
The strongest form of the hot spots conjecture states 
\begin{conjecture}\label{conj}
The global maximum and global minimum for any eigenfunction belonging to $\mu_2$ is not attained inside the domain. 
\end{conjecture}
For discussion on various other formulations see the work of Ba\~nuelos and Burdzy \cite{BB99}. The conjecture is false for arbitrary sets as shown by Burdzy and Werner \cite{BW99}, but it is most likely true for convex domains. For the overview of the known results and counterexamples see \autoref{sechist}.

In this paper we are concerned with triangular domains. The conjecture is known only for obtuse, right and isosceles triangles. Surprisingly, it was open for any nonsymmetric acute triangle. Recently, acute triangles became the main subject of the Polymath7 project \cite{poly}, whose main goal is to give a proof in this special case, as well as establish new numerical and analytical tools for studying this and related problems.

We refine the method of boundary critical points used by Miyamoto \cite{Mi12} to give an analytic proof of the hot spots conjecture for isosceles triangles. As a result we resolve the conjecture for triangles with one small angle via a reduction to a problem on the boundary of the domain.

\subsection{Main results}
First we establish simplicity for $\mu_2$. This allows us to work with an essentially unique second eigenfunction. Simplicity is also important for numerical stability and theoretical estimates for eigenfunctions needed in the Polymath7 project \cite{poly}. We also use it in symmetry based arguments for triangles and kites.
\begin{theorem}\label{simple}
The second Neumann eigenvalue $\mu_2$ is simple for all non-equilateral triangles.
\end{theorem}
This result was already known for obtuse and right triangles \cite{AB04}, and for isosceles triangles \cite{Mi12}.
The proof is rather lengthy and we postpone it to the last section.

Next we reduce the hot spots problem to a simpler problem on the boundary of a triangle.
\begin{theorem}\label{2sides}
 If the second Neumann eigenfunction has no critical points on two sides of an acute triangle, than the hot spots conjecture holds.
\end{theorem}

As a consequence we get a partial proof of the conjecture.
\begin{theorem}\label{pi6}
  The hot spots \autoref{conj} holds for acute triangles with an angle smaller then or equal to $\pi/6$.
\end{theorem}
In fact in \autoref{sechotspots} we prove the conjecture for these, and a few more triangles. This result relies on two key lemmas: symmetry of the second Neumann eigenfunction of kites, and a critical point elimination lemma. 

Let $T$ be a triangle with vertices $(0,0)$, $(1,0)$ and $(a,b)$ with $0\le a\le 1/2$ and $b>0$. Let $K$ be the kite obtained by mirroring the triangle along the $x$-axis (the fourth vertex is then $(a,-b)$). We have 
\begin{lemma}\label{kitesym}
  If $3b^2\le 1-a+a^2$ then the second Neumann eigenfunction of $K$ is simple and symmetric with respect to the $x$-axis, except for the square ($a=b=1/2$) and the equilateral triangle (degenerate kite with $a=0$ and $b=1/\sqrt{3}$). In these cases $K$ has a double eigenvalue ($\mu_2=\mu_3$).
\end{lemma}

The above lemma is used in conjunction with the following lemma to prove that there are no critical points on two sides of the triangle $T$.
\begin{lemma}[Generalization of Miyamoto {\cite[Lemma 3.2]{Mi12}}]\label{critical}
  Suppose that $\mu_2(T)\le \frac{\pi^2}{b^2}$ and $\varphi$ is the eigenfunction for $\mu_2$. Then
  \begin{itemize}
    \item $\varphi$ has at most one critical point on the side on the $x$-axis. This critical point (if it exists) is a positive minimum or a negative maximum. 
    \item If the second Neumann eigenfunction of $K$ is simple and symmetric with respect to the $x$-axis, then $\varphi$ has no critical point on the side along the $x$-axis and it is changing sign along this side. 
  \end{itemize}
\end{lemma}

However to apply this lemma we need a bound for $\mu_2(T)$. 
\begin{lemma}\label{mubound}
  If $b^2\le a^2+(1-a)^2$ then
  \begin{align}\label{muineq}
    \mu_2(T)\le \frac{\pi^2}{b^2}.
  \end{align}
  In particular, the condition \eqref{muineq} holds for any triangle with the longest or middle side on the $x$-axis and the angle up to $\pi/4$ at vertex $(1,0)$. 
\end{lemma}

For the discussion on how all these results can be applied to the hot spots problem, see \autoref{sechotspots}, in which we prove \autoref{pi6}. 

  \subsection{History of the problem}\label{sechist}
  The hot spots conjecture was posed by J. Rauch in 1975 \cite{R75} for arbitrary open sets. The first positive result was obtained by Kawohl \cite{K85} for products of an arbitrary domain and an interval. In the same manuscript he also restates the conjecture just for convex sets. Subsequent counterexamples by Burdzy and Werner \cite{BW99} (two holes) and Burdzy \cite{B05} (one hole) show that the restriction to convex domains might be necessary. 
 
  The hot spots conjecture for convex domains remains open, however many special cases have been solved. Ba\~nuelos and Burdzy were able to handle domains with a line of symmetry and a few more (quite restrictive) technical assumptions \cite{BB99}. A year later Jerison and Nadirashvili \cite{JN00} proved that the conjecture holds for domains with two lines of symmetry. In a different direction, Burdzy and Atar \cite{AB04} had to assume that the domain is bounded by graphs of two Lipschitz functions with Lipschitz constant 1.

  All known results assume some degree of symmetry or special shape of the boundary. Surprisingly, domains as simple as acute triangles are not covered by any known result (note that obtuse and right triangles were solved \cite{BB99,AB04}). The conjecture for isosceles triangles can be obtained by combining \cite{AB04,P02} and \cite{LSminN}, or directly using the new method due to Miyamoto \cite{Mi12}. Refinement of this new method leads to the results of this paper. There is also an active Polymath7 project \cite{poly} proposed by Chris Evans and moderated by Terrence Tao. The current focus of the project is on developing robust validated numerical methods that would lead to the proof of the hot spots conjecture for acute triangles and possibly other domains.

\section{Preliminary results}
\subsection{Symmetric modes}
Note that on a domain with a line of symmetry, the second Neumann eigenfunction is either symmetric or antisymmetric. In case the eigenvalue is not simple, we can decompose any eigenfunction into symmetric and antisymmetric parts. 

Any symmetric mode on a symmetric domain satisfies Neumann condition on the line of symmetry. Hence it is also a mode for the half of the domain. Therefore the lowest symmetric mode on the symmetric domain must be the same as the eigenfunction for $\mu_2$ for each half. Note however, that this symmetric mode does not need to belong to $\mu_2$ on the whole domain. We need the following stronger results for symmetric modes.

\begin{figure}[t]
  \begin{center}
\begin{tikzpicture}[scale=1]
  \draw[clip] (0,0) .. controls (-1,2) and (2,3) .. (3,0) .. controls  (2,-3) and (-1,-2).. (0,0); 
  \draw[dashed] (0,0) -- (3,0);
  \draw[dashed] (0.5,2) .. controls (2,0) .. (0.5,-2);
  \draw (0.5,1.2) node {$D^+$};
  \draw (0.5,-1.2) node {$D^-$};
\end{tikzpicture}
  \end{center}
  \caption{The second antisymmetric mode. Dashed curves denote nodal lines.}
  \label{figanti}
\end{figure}
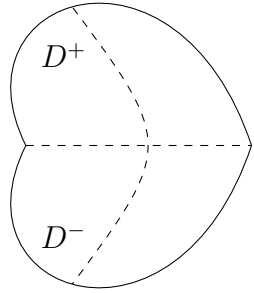

\begin{lemma}\label{symmetric}
  Suppose $D$ is a domain with a line of symmetry. Then there cannot be two orthogonal antisymmetric eigenfunctions in the span of the eigenspaces of $\mu_2$ and $\mu_3$ (note that $\mu_2$ might equal $\mu_3$).

  This means that either $\mu_2$ or $\mu_3$ must have a symmetric eigenfunction. It is also possible that all eigenfunctions for these eigenvalues are symmetric, as is the case for narrow subequilateral triangles, or narrow sectors.
\end{lemma}
\begin{proof}
  Suppose that the line of symmetry divides $D$ into $D^+$ and $D^-$, see \autoref{figanti}.
  Suppose also there are two orthogonal antisymmetric eigenfunctions in the span of the eigenspaces of $\mu_2$ and $\mu_3$. One of them must change sign in $D^+$ (and by antisymmetry in $D^-$), otherwise these eigenfunctions would not be orthogonal. This particular eigenfunction will have at least 4 nodal domains, contradicting Courant's nodal domain theorem.
\end{proof}

Let $\lambda_1(D)$ be the smallest Dirichlet eigenvalue of $D$. To prove the next result we need the following eigenvalue comparison result
\begin{theorem}[Friedlander '95]\label{fri95}
  For convex domains
   \begin{align*}
     \lambda_1(D)\ge \mu_3(D).
   \end{align*}
\end{theorem}

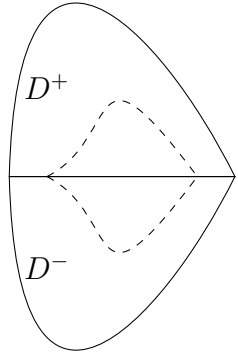
\begin{figure}[t]
  \begin{center}
\begin{tikzpicture}[scale=1]
  \draw (0,0) .. controls (0.1,4) and (2,2) .. (3,0) -- cycle; 
  \draw (0,0) .. controls (0.1,-4) and (2,-2) .. (3,0) -- cycle; 
  \draw[dashed] (0.5,0) .. controls (1.5,0.5) and (1,2) .. (2.5,0);
  \draw[dashed] (0.5,0) .. controls (1.5,-0.5) and (1,-2) .. (2.5,0);
  \draw (0.5,1.2) node {$D^+$};
  \draw (0.5,-1.2) node {$D^-$};
\end{tikzpicture}
  \end{center}
  \caption{Nodal line cannot start and end on the straight part of the boundary (as shown).}
  \label{figstr}
\end{figure}

\begin{lemma}[Polymath7 \cite{poly}] \label{line}
  Suppose that we have a convex domain $D^+$ with a straight part of the boundary, such that the domain $D$ obtained by mirroring about the straight part is also convex, see \autoref{figstr}.
  The second Neumann eigenvalue of $D^+$ cannot have an eigenfunction with nodal line starting and ending on this straight piece of the boundary (including the endpoints of the straight piece).
\end{lemma}
\begin{proof}
Suppose the second eigenvalue $\mu_2(D^+)$ has an eigenfunction $\varphi$ with nodal line starting and ending on the same straight piece of the boundary. 
We can unfold the domain $D^+$ and $\varphi$ to get a symmetric domain $S=D^+\cup D^-$ and its symmetric eigenfunction with closed nodal domain $N$ inside. The Dirichlet eigenvalue of this nodal domain ($\lambda_1(N)=\mu_2(D^+)$) is strictly larger than the first Dirichlet eigenvalue $\lambda_1(S)$. This one is however larger than or equal to $\mu_3(S)$ (\autoref{fri95}). Hence 
\begin{align}\label{ineq}
\mu_2(D^+)=\lambda_1(N)>\lambda_1(S)\ge \mu_3(S).
\end{align}

Eigenfunction $\varphi$ is the lowest symmetric mode of $S$, since it belongs to the smallest positive eigenvalue on $D^+$. By \autoref{symmetric} it must belong to either $\mu_2(S)$ or $\mu_3(S)$. In either case we get a contradiction with \eqref{ineq}.
\end{proof}
\newcommand{\regular}[1]{
\pgfmathsetmacro{\n}{#1}
\pgfmathsetmacro{\a}{360/\n}
\coordinate (a) at (0,0);
\foreach \x in {0,\a,...,360} 
{
   \draw (a) -- (\x:1) coordinate (a) -- (0,0);
}
\clip (0,0) -- (0:1) -- (\a:1) -- cycle;
\draw[fill=gray, opacity=0.3] (0,0) -- (\a/2:1) -- (1,0) -- cycle;;
}

  \subsection{Nodal line approach}
  In this section we collect the results needed for the approach due Miyamoto \cite{Mi12}. However, we generalize most of the key lemmas to avoid the symmetry assumptions for triangles.
 
  First we need the following consequence of real analyticity for eigenfunctions
  \begin{lemma}[{\cite[Corollary 2.2]{Mi12}}]\label{branches}
    Suppose $u$ satisfies $\Delta u=-\mu u$ on $D$ (without any boundary condition). If $u(x,y)=u_x(x,y)=u_y(x,y)=0$ (degenerate zero) then either $u\equiv 0$ or $\{u=0\}$ has at least 4 branches from $(x,y)$ and $\{u>0\}$ (and $\{u<0\}$) has at least 2 connected components near $(x,y)$ (but these might be globally connected).
  \end{lemma}
  
  For an arbitrary function $u$ on $D$, the nodal set $\{u=0\}$ may contain a loop. More precisely, there might exist a nodal domain with the boundary contained in $\{u=0\}$. However, eigenfunctions often have no loops. We generalize \cite[Lemma 2.3]{Mi12} (for convex domains) using \autoref{fri95}.
  \begin{lemma}\label{loop}
    Let $D$ be a convex domain, and $u$ be any function satisfying $\Delta u=-\mu u$ on $D$ (without boundary conditions). If $\mu\le \mu_3(D)$, then $\{u=0\}$ has no loop in $\overline{D}$.
    
    In particular nodal lines of partial derivatives of the first two eigenfunctions cannot have loops.
  \end{lemma}
  \begin{proof}
    Suppose there is a loop and let $F$ be the set enclosed by the loop. Then
    \begin{align*}
      \mu=\lambda_1(F)>\lambda_1(D)\ge \mu_3(D)\ge \mu.
    \end{align*}
    Giving contradiction.
  \end{proof}  
  We can also strengthen the first part of \cite[Lemma 2.4]{Mi12}.
  \begin{lemma}\label{nodegzero}
    If $u$ is an eigenfunction for convex $D$ belonging to $\mu_2$ or $\mu_3$ then $u$ does not have a degenerate zero in $D$.
  \end{lemma}
  \begin{proof}
    Degenerate zero implies at least 4 branches for $\{u=0\}$. Therefore locally we have two nodal domains where eigenfunction is positive, and between them there are two domains with negative sign. If the two positive nodal domains were globally connected, then there would be a curve that connects a point near the critical point from one of them to a point in the other. Hence the negative nodal domain between the two positive subdomains is closed inside the original domain. Hence the negative nodal domain forms a closed loop as part of the nodal set, contradicting \autoref{loop}. Hence the positive nodal domains near the critical point are not globally connected, similarly the negative nodal domains. This contradicts Courant's nodal domain theorem, since we have at least 4 nodal domains.
  \end{proof}
Define
  \begin{align*}
    \mathcal H_\mu[u]=\int_D (|\nabla u|^2-\mu u^2)dA
  \end{align*}
  Then by variational formula for eigenvalues
  \begin{lemma}\label{Hcond} 

    If $\int_D u=0$, then $\mathcal H_{\mu_2}[u]\ge 0$.

    If $\int_D u=0$ and $u$ is symmetric, then $\mathcal H_{\mu_s}[u]\ge 0$. Here $\mu_s$ is the lowest symmetric mode for a symmetric $D$.

    In general whenever $u$ is a valid test function for $\lambda$, then $\mathcal H_\lambda[u]\ge 0$. Here $\lambda$ can be any type of eigenvalue.
  \end{lemma}
  This observation was used by Miyamoto \cite{Mi12} in contradiction arguments (one needs to construct $u$ such that $\mathcal H[u]<0$).
  Suppose that $\Delta u=-\mu u$ on $D$ (no boundary conditions).
  Then 
  \begin{align*}
    \mathcal H_\mu[u]=\int_{\partial D} u\partial_\nu ud\sigma.
  \end{align*}
  In particular for any mixed Dirichlet-Neumann boundary conditions $\mathcal H_\mu[u]=0$. However, one can get a contradiction in the above lemma by controlling the sign of the product $u\partial_\nu u$.

  Note that in \autoref{line} we can drop the assumption that we have an eigenfunction, cf. \autoref{loop}.
\begin{lemma}\label{line2}
  Let $D$ be a convex domain with a straight piece of boundary, $\Delta u=-\mu_2 u$ and $\partial_\nu u\le 0$ on the straight piece whenever $u>0$. Then $\partial\{u>0\}$ cannot have a connected component bounded by a curve starting and ending on the straight piece of the boundary.
\end{lemma}
\begin{proof}
  If this was the case then $H_{\mu_2}[u]\le 0$ on this connected component and $\mu_2$ would be larger than or equal to the mixed Dirichlet-Neumann eigenvalue for the component. Now we can apply the argument from the proof of \autoref{line}.
\end{proof}

\section{Proofs of the main results}

\subsection{hot spots conjecture}\label{sechotspots}
Here we prove a result stronger than \autoref{pi6}. 

Note that the condition in \autoref{kitesym} 
\begin{align}\label{kitecond}
  3b^2<1-a+a^2
\end{align}
implies the condition in \autoref{mubound}
\begin{align}\label{mucond}
  b^2<a^2+(1-a)^2.
\end{align}
Therefore as soon as we can apply \autoref{kitesym}, we also have eigenvalue bound from \autoref{mubound}. Therefore \autoref{critical} implies no critical point on the side on the $x$-axis. Therefore the hot spots conjecture is true for any triangle for which the former is satisfied for two ways to put the triangle in the coordinate system (the longest or the middle side on the $x$-axis), by \autoref{2sides}. 

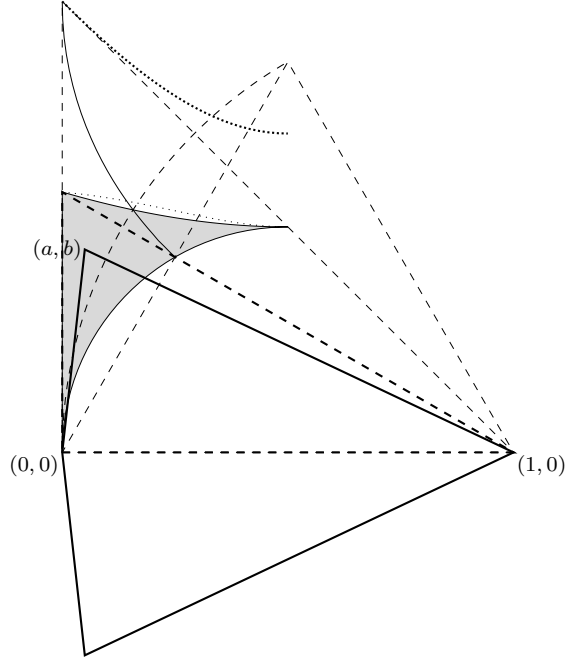
\begin{figure}[t]
  \begin{center}
  \begin{tikzpicture}[scale=6]
    \fill[gray,fill opacity=0.3,draw=black,very thin] (0,0) arc (180:90:0.5) -- plot [domain=0.5:0] (\x,{sqrt((1-\x+\x^2)/3)}) -- cycle;
    \draw[very thin] plot  [domain=0:0.5] ({(1-2*\x)*(1-2*\x)/(4*\x*\x-7*\x+4)},{sqrt(3*(\x*\x-\x+1))/(4*\x*\x-7*\x+4});
    \draw[densely dotted,thick] plot  [domain=0:0.5] (\x,{sqrt(\x*\x+(1-\x)*(1-\x))});
    \draw[very thin,dashed] (1,0) -| (0,1) -- cycle;
    \draw[very thin,dashed] (0,0) -- (1,0) -- (1/2,{sqrt(3)/2}) -- cycle;
    \draw[thick,dashed] (0,0) -- (1,0) -- (0,{1/sqrt(3)}) -- cycle;
    \draw[thick] (0,0) -- (0.05,0.45) -- (1,0) -- (0.05,-0.45) -- cycle;
    \draw[dashed] (0,0) arc (180:120:1);
    \draw[dotted,smooth] plot coordinates {(0,{1/sqrt(3)})(0.05,0.57180)(0.1,0.56458)(0.15,0.55592)(0.2,0.54618)(0.25,0.53578)(0.3,0.52531)(0.35,0.51554)(0.4,0.50740)(0.45,0.50194)(1/2,1/2)};
    \draw (0,0) node [below left=-3pt] {\tiny $(0,0)$};
    \draw (1,0) node [below right=-3pt] {\tiny $(1,0)$};
    \draw (0.05,0.45) node [left=-3pt] {\tiny $(a,b)$};
  \end{tikzpicture}
  \end{center}
  \caption{Triangle $T(a,b)$, its kite and various conditions on $a$ and $b$.}
  \label{fig1}
\end{figure}

Before we finish the proof of \autoref{pi6} we discuss \autoref{fig1}. The dashed straight lines form 3 triangles: equilateral, right isosceles and a half-equilateral right triangle (angle $\pi/6$ near $(1,0)$). The dashed arc gives $(a,b)$ pairs for isosceles triangles. Below this line the longest side is on the $x$-axis, above it the middle (or the shortest) side. Hence all triangles can be uniquely described by a pair $(a,b)$ with $(1-a)^2+b^2\le 1$, while acute triangles satisfy $a^2+b^2>a$ (lower boundary of the gray area). The gray area contains all $(a,b)$ pairs for which kite $K$ has simple and symmetric second eigenfunction according to \eqref{kitecond}, while a dotted line just above the area is a numerical curve on which the symmetric mode equals the antisymmetric mode. The thick dotted line depicts the boundary of the region given by condition \eqref{mucond} from \autoref{mubound}.

Finally, the solid curve is the inversion of the upper part of the boundary of the gray area with respect to the ``isosceles circle'' $(a-1)^2+b^2=1$. It happens that the inversion of $(a,b)$ gives a new placement for the same triangle (the middle side interchanges with the longest). Indeed, if $(a,b)$ is inside the ``isosceles circle'', then the longest side of the triangle is on the $x$-axis and has length $1$. The middle side has vertices $(a,b)$ and $(1,0)$. If we invert the point $(a,b)$ with respect to the ``isosceles circle'', then the ratio between the longest and the middle side does not change, however the middle side is now on the $x$-axis. Therefore rescaling to get the longest side of length $1$ leads to the same triangle as the original triangle (same two sides and the angle between). 

Hence the hot spots conjecture is true for any triangle in the gray area and below the solid curve. This clearly contains all triangles with the smallest angle up to $\pi/6$ (below the thick dashed line), implying \autoref{pi6}.
\subsection{Proof of \autoref{critical}}

Let $T$ be the triangle $OAB$ and $K$ be the kite $OBAB'$ (see \autoref{figcrit}). We can assume that $|OA|=1$. Let $\varphi$ be the eigenfunction for $\mu=\mu_2(T)$. It is also the lowest symmetric mode for $K$, and it belongs to either $\mu_2(K)$ or $\mu_3(K)$ (by \autoref{symmetric}). Suppose $\varphi$ has a critical point $p$ on $OA$. It cannot be $0$ there, by \autoref{nodegzero}. Without loss of generality we can assume that $\varphi(p)>0$. If there is more than one critical point take the one with maximal value of $\varphi(p)$. Let
\begin{align*}
  \psi(x,y)=\cos(\sqrt{\mu}y).
\end{align*}
By assumption $\mu y^2\le \mu b^2\le \pi$, hence $\psi_y(x,y)<0$ when $y>0$. Therefore outward normal derivative $\partial_\nu \psi<0$ on the boundary of $K$.  
\begin{figure}[t]
  \begin{center}
\begin{tikzpicture}
  \draw (-2,0) -- (2,0);
  \draw[very thick] (-2,0)  -- (1,2.5) -- (2,0)  -- (1,-2.5) -- cycle;
  \draw (-2,0) node [below] {$O$};
  \draw (2,0) node [below] {$A$};
  \draw (1,2.5) node [right] {$B$};
  \draw (1,-2.5) node [right] {$B'$};
  \fill (0.5,0) circle (1pt) node [below] {$p$};
  \clip (-2,0)  -- (1,2.5) -- (2,0)  -- (1,-2.5) -- cycle;
  \draw[dashed] (0.5,0) .. controls (0.,0.7) and (-1,0.5) .. (-1,1);
  \draw[dashed] (0.5,0) .. controls (0.,-0.7) and (-1,-0.5) .. (-1,-1);
  \draw[dashed] (0.5,0) .. controls (1,0.5) and (0.5,2) .. (2,2);
  \draw[dashed] (0.5,0) .. controls (1,-0.5) and (0.5,-2) .. (2,-2);
  \draw[] (0.3,1) node {$F_1$};
  \draw[] (0.3,-1) node {$F_2$};
\end{tikzpicture}
  \end{center}
  \caption{Four branches from a critical point $p$ inside a kite.}
  \label{figcrit}
\end{figure}
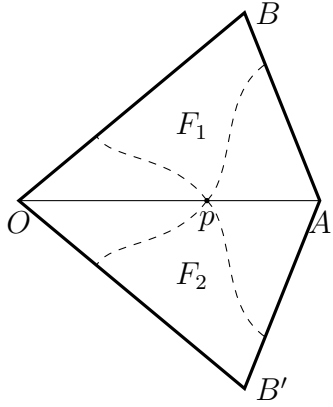

Take $u(x,y)=\varphi(p)\psi(x,y)-\varphi(x,y)$. Then $u(p)=u_x(p)=u_y(p)=0$ (degenerate zero). Hence there are four branches of $\{u=0\}$ around $p$ (by \autoref{branches}), unless $u$ is the eigenfunction, but it does not satisfy Neumann boundary condition. Furthermore $\{u>0\}$ has at least 2 connected components near $p$ and these cannot be globally connected since $\{u=0\}$ has no loops. Therefore there are at least two disjoint nodal domains $F_1$ and $F_2$ of $K$ such that $u\ge 0$ on $\overline{F_i}$. Finally $\partial_\nu u=\partial_\nu \psi<0$ on $\partial K$.

Suppose $u\le 0$ on $OA$, then $\varphi(x,y)\ge \varphi(p)>0$ on $OA$, and the eigenfunction is strictly positive on $OA$. Furthermore all points such that $\varphi(x,y)=\varphi(p)$ are also critical points for the side and $u$ is zero there. We will eliminate the possibility of 2 degenerate zeros later. Note also that if we had two critical points with different value of $\varphi$, then we have taken the larger one as $p$, hence $u$ will be positive somewhere and this case does not apply.

Suppose $u>0$ somewhere on $OA$. Then at least one of $F_i$ must contain a part of $OA$ and it must be symmetric with respect to $OA$. Suppose $F_1$ has this property. Take $G=\{u>0\}\setminus F_1$, then $G$ is also symmetric, since $F_1\cup G$ is symmetric.

Define a symmetric test function $v$
\begin{align*}
  v=u 1_{F_1}-cu 1_{G},
\end{align*}
where $c$ is chosen so that $\int_K v=0$. 

This is a valid test function for $\mu$ (regardless if it equals $\mu_2(K)$ or $\mu_3(K)$). Note that
\begin{align*}
  v\partial_\nu v=(u\partial_\nu u)1_{F_1}+c^2(u\partial_\nu u)1_{F_2}\le 0,\text{ on } \partial(F_1\cup G).
\end{align*}
Indeed, either $u=0$ or $u>0$ and outward normal is negative. Therefore $\mathcal H_\mu[v]\le 0$. But it cannot be equal $0$ since $v$ equals $0$ on an open set, and it cannot be the eigenfunction. This contradicts \autoref{Hcond}.

We already showed that if $u\le 0$ on $OA$ then we have a global minimum, possibly at two or more points. However this means two or more degenerate zeros for $u$. In this case degenerate zeros generate disjoint sets $\{u>0\}$, since there are no loops. Take these sets as $F_1$ and $G$ to get a contradiction. Hence there is only one global minimum.

Finally if $\mu_2(K)$ has symmetric eigenfunction we do not need symmetry of $v$ and we can take any $F_1$ and $F_2$ in its definition. This proves that even if $u\le0$ on $OA$, we still cannot have a critical point. Moreover, since $\mu_2(K)$ has symmetric eigenfunction, this eigenfunction must be $0$ somewhere on the line of symmetry, otherwise we would have at least 3 nodal domains.

\subsection{Proof of \autoref{2sides}}

\begin{figure}[t]
  \begin{center}
\begin{tikzpicture}[scale=2]
  \draw (0,0) node [below] {$A$} -- (2,0.4) node [below] {$B$} -- (0,1) node [above] {$C$} -- cycle; 
  \draw[->,thick] (1,0.2) -- ++(0.4,0.4/5);
  \draw[->,thick] (1,0.2) -- ++(0.4,0);
  \draw[->,thick] (1,0.2) -- ++(0,0.4/5);
  \draw[->,dashed] (1,0.2) -- ++(0.04,-0.2) node [below=-3pt] {\tiny $u_\nu=0$};
\end{tikzpicture}
  \end{center}
  \caption{Acute triangle with one vertical side and no critical points on sloped sides.}
  \label{figts}
\end{figure}
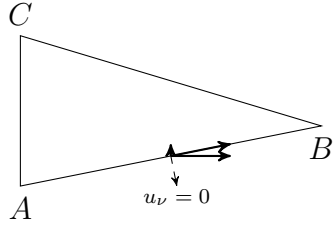

Position the triangle as on \autoref{figts}. Since this triangle is acute, the bottom side is sloped up, and the upper side is sloped down. Let $u$ be the second Neumann eigenfunction of $ABC$. We know that there are no critical points on two sloped sides, and we have Neumann boundary conditions there. Therefore $u_x$ and $u_y$ cannot change sign on these sides and they are never $0$ there. Note also that $u_x=0$ on $AC$, $u_x$ and $u_y$ must have the same signs on $AB$, and opposite signs on $BC$.

If $u_x>0$ (or $u_x<0$) on both $AB$ and $CB$, then $u_x\ge 0$ on the boundary of the triangle. If there was a point $p$ inside $ABC$ such that $u_x(p)=0$, then real analyticity implies that $u_x<0$ at some point near $p$. Therefore $u_x<0$ would have to form an open nonempty subset inside (possibly with a piece of the boundary on $AC$). Hence $u_x=0$ would have a loop, contradicting \autoref{loop}. Therefore $u_x>0$ inside $ABC$. Therefore the global maximum and the global minimum of $u$ must be on the boundary. One of them must be at $B$, the other on $AC$. This is the case for subequilateral triangles.

   If $u_x>0$ on $AB$ and $u_x<0$ on $CB$, then $u_y>0$ on $AB$ and on $CB$ (similar argument for opposite signs). Furthermore $u_y$ satisfies Neumann boundary condition on $AC$, since $(u_y)_x=(u_x)_y=0$ on $AC$. As before $\{u_y<0\}$ would need to form a nonempty subset of $T$, possibly with a part of the boundary on $AC$. But this contradicts \autoref{line2}. Hence $u_y>0$ on $T$. Hence the maximum is at $C$ and the minimum at $A$. This is the case for superequilateral triangles.

As a consequence we obtain
\begin{corollary}[Atar, Burdzy \cite{AB04}]
  hot spots conjecture holds for the lowest symmetric mode of any acute isosceles triangle (note that for superequilateral triangles this is the third eigenfunction). This in turn implies that the conjecture holds for all right triangles.

  Note that for subequilateral triangles and the corresponding right triangles our proof follows closely Miyamoto's proof of the same result.
\end{corollary}
\begin{proof}
For right triangles with the longest side of length $1$ and the shortest altitude $b$, Theorem 3.1 from \cite{LSmaxN} gives
\begin{align*}
  \mu_2 b^2\le \frac{4\pi^2 b^2}{3\sqrt{3}A}\le\frac{4\pi^2 b^2}{3\sqrt{3}b^2}=\frac{4\pi^2}{3\sqrt{3}}\le \pi^2.
\end{align*}
Hence we can apply \autoref{critical} to the half of the acute isosceles triangle mirrored along the longest side. We prove below that corresponding kite has symmetric eigenfunction hence there are no critical points on the equal sides of the isosceles triangle. Now we apply \autoref{2sides}.
\end{proof}

\subsection{Proof of \autoref{mubound}}

 Start with the second eigenfunctions for two right isosceles triangles with $(a,b)=(0,1)$ and $(a,b)=(1/2,1/2)$. These are
 \begin{align*}
   \varphi_1(x,y)=\cos(\pi y)-\cos(\pi x),\\
   \varphi_2(x,y)=\cos(\pi x)\cos(\pi y).
 \end{align*}
 The only property of these functions we actually need is that they integrate to $0$ over their respective right isosceles triangles (orthogonal to constants). Now we apply linear transformations to obtain functions on $T(a,b)$ with arbitrary $(a,b)$. Note that we retain orthogonality to constants. Take
 \begin{align*}
   f(x,y)=(1/2-a)\varphi_1(x-ay/b,y/b)-a \varphi_2(x+(1-2a)y/2b,y/2b).
 \end{align*}
 This is a valid test function for $\mu_2(T)$ (it integrates to $0$ over $T(a,b)$). Note that when $a=0$ or $a=1/2$ we recover the exact eigenfunctions for the right isosceles triangles we considered. Let $c=a(a-1)$. We get
 \begin{align*}
   \mu_2(T)\le \frac{\pi^2(2c(2+b^2+c)+b^2+1)-16c(b^2+c)}{2(3c+1)b^2}\stackrel{?}{\le}\frac{\pi^2}{b^2},
 \end{align*}
 where we need to prove the last inequality. Hence we need
 \begin{align*}
  \pi^2(2c(2+b^2+c)+b^2+1)-16c(b^2+c)-2\pi^2(3c+1)\le 0.
 \end{align*}
 Put $d=b^2+c$. Now we need
 \begin{align*}
   0\ge \pi^2(2cd+d+1)-16cd-\pi^2(3c+2)=d(\pi^2(1+2c)-16c)-\pi^2(3c+1)
 \end{align*}
 Note that $0\ge c\ge -1/4$, hence the coefficient in front of $d$ is positive. But
 \begin{align*}
   d=b^2+a^2-a\le a^2+(1-a)^2+a^2-a=3a^2-3a+1=3c+1.
 \end{align*}
 Hence the desired inequality is true if
 \begin{align*}
   0\ge (3c+1)(\pi^2(1+2c)-16c)-\pi^2(3c+1)=(3c+1)c(2\pi^2-16)
 \end{align*}
 But $3c+1>0$, $c\le 0$ and $2\pi^2-16>0$. Hence the inequality is true.

 \section{Kites and \autoref{kitesym}}\label{kitesec}
 Recall that for a triangle $T(a,b)$ we define a kite $K$ by mirroring the triangle with respect to the $x$-axis. We consider the lowest antisymmetric modes and their eigenvalues $\mu_a(K)$. We prove that if $3b^2\le 1-a+a^2$, then the eigenvalues $\mu_a$ are above the second Neumann eigenvalue. This ensures that all eigenfunctions for $\mu_2(K)$ are symmetric. But then they are also eigenfunctions for $T(a,b)$ (with simple eigenvalue). Hence these kites have simple second eigenvalue. This proves \autoref{kitesym}.

 Let $\mu_a$ be the lowest antisymmetric mode on $K$. Then $\mu_a$ is the lowest eigenvalue of the mixed Dirichlet-Neumann problem on $T(a,b)$ with Dirichlet condition on the $x$-axis. We find lower bound for this eigenvalue using unknown trial function method developed in \cite{LSminN} and \cite{LSminD}.

 Then we find an upper bound for $\mu_2(T)\ge \mu_2(K)$ that is smaller than the lower bound from the first step.
  \subsection{Lower bound for $\mu_a$}$\;$
 
  Let $\lambda(a,b)=\mu_a(a,b)$ be the lowest eigenvalue of the mixed problem on $T(a,b)$ with Dirichlet condition on the $x$-axis and Neumann on the other two sides.
  We will use the following unknown trial function lemma.
  \begin{lemma}[Laugeen and Siudeja {\cite[Lemma 4.1]{LSminD}}]\label{unknown}
    Let $\lambda(a,b)$ be the eigenvalue for a triangle with vertices $(\pm1,0)$ and $(a,b)$. The inequality
    \begin{align*}
      \lambda(a,b)\ge C_{a,b,c,d}\lambda(c,d)
    \end{align*}
    is true if
    \begin{align*}
      ( (a-c)^2+d^2)(1-\gamma)+2b(a-c)\delta+b^2\gamma\le d^2/C_{a,b,c,d},
    \end{align*}
    where $\delta$ and $\gamma$ are some numbers (unfortunately unknown) depending only on $a$ and $b$ and satisfying $|\delta|\le 1/2$ and $0\le\gamma\le 1$.
  \end{lemma}
  \begin{remark}
    This lemma relies on linear transformation between triangles. However the result holds for any family of domains that can be obtained using the same linear transformation. In particular, the same is true for triangles $T(a,b)$ (with vertices $(0,0)$, $(1,0)$ and $(a,b)$). Furthermore, this lemma applies to any mixed boundary conditions (see also \cite[Corollary 5.5]{LSminN}).
  \end{remark}
  \begin{remark}
    In order to use this inequality we would need to prove the ``if'' part for any $\gamma$ and $\delta$. Instead, we can choose a few sets of values of $c$ and $d$ so the eigenvalues on the right are explicit, effectively obtaining a few inequalities involving $a, b, \gamma, \delta$. For fixed $a$ and $b$ we need to show that at least one of those inequalities is true for any admissible pair $(\gamma,\delta)$. 
  \end{remark}

  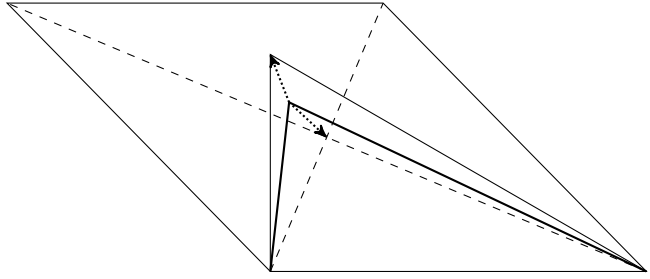
\begin{figure}[t]
    \begin{center}
  \begin{tikzpicture}[scale=5]
    \draw[thick] (1,0) -- (0.05,0.45) coordinate (a)
    -- (0,0);
    \draw[] (1,0) -- (0,{1/sqrt(3)}) coordinate (b) -- (0,0);
    \draw[densely dotted,thick,->] (a) -- (b);
    \draw[densely dotted,thick,->] (a) -- (0.15,{sqrt(0.15*(1-0.15)}) coordinate (d);
    \draw[dashed] (0,0) -- ($(0,0)!2!(d)$) coordinate (dd);
    \draw[] (0,0) -- (1,0) -- (dd) -- ++(-1,0) coordinate (e) -- cycle;
    \draw[dashed] (1,0) -- (e);
  \end{tikzpicture}

    \end{center}
    \caption{A triangle mapped onto a half-equilateral triangle or a quarter of a rhombus}
    \label{figrh}
  \end{figure}

  We consider two pairs $(c,d)$.
  \begin{enumerate}
    \item $(c,d)=(0,1/\sqrt{3})$ (half-equilateral triangle on \autoref{figrh}). On this triangle $\lambda(0,1/\sqrt{3})=4\pi^2/3$ (the lowest Neumann eigenvalue of the equilateral triangle with side length 4/3, see e.g. McCartin \cite{McC02}).
    \item $(c,d)=(c,\sqrt{c-c^2}=:h)$, where $a\le c\le 1/2$ will be chosen later. Here we get $\lambda(c,d)=\lambda_1(R)$ (the first Dirichlet eigenvalue of the rhombus, see \autoref{figrh}). We can use Hooker-Protter bound \cite{HP61} to get 
      \begin{align*}
	\lambda(c,h)\ge \frac{\pi^2(1+2h)}{4h^2}.
      \end{align*}
      See also \cite{FS10} for comparisons of known bounds for rhombi, showing that Hooker-Protter bound is the best for relatively square rhombi.
  \end{enumerate}
   
  Suppose we want to prove
  \begin{align*}
    \mu_a( (a,b))=\lambda(a,b)\ge \frac{4\pi^2}{3F},
  \end{align*}
  for some, not yet known $F=F(a,b)$.
  \begin{enumerate}
    \item 
      First consider $c=0$ and $d=1/\sqrt{3}$. To get the bound we want we need
      \begin{align}\label{transeq}
      (3a^2+1)(1-\gamma)+6ab\delta+3b^2\gamma\stackrel{?}{\le} F,
    \end{align}

  \item Using $d=h=\sqrt{c-c^2}$ we need
    \begin{align}\label{transrhombus}
      ( (a-c)^2+h^2)(1-\gamma)+2b(a-c)\delta+b^2\gamma\stackrel{?}{\le} \frac{3(1+2h)F}{16}
    \end{align}
  \end{enumerate}
  We need to show that at least one of the inequalities (\ref{transeq},\ref{transrhombus}) is true. We can achieve that by proving that one positive linear combination of those inequalities is true. We can choose this linear combination so that $\delta$ cancel.

  Therefore we combine $c-a$ times \eqref{transeq} and $3a$ times \eqref{transrhombus} and we need to prove
  \begin{align}\label{combined}
    (3ca(1-a)+c-a)(1-\gamma)+3b^2 c \gamma
    \stackrel{?}{\le}
    \left(c-a+\frac{9a(1+2h)}{16}\right)F.
  \end{align}
  This inequality must be true for any $0\le \gamma\le 1$. To simplify the task we may choose $c$ so that the expressions in front of $1-\gamma$ and $\gamma$ are equal, effectively eliminating $\gamma$. That is
  \begin{align*}
    c=\frac{a}{1-3\delta},
  \end{align*}
  with $\delta=a^2+b^2-a\ge0$. Note that $c=a$ for all right triangles ($\delta=0$), hence $T(c,h)=T(a,b)$ for right triangles and we are using the Hooker-Protter bound for all right triangles. This gives 
  \begin{align*}
    \frac{3b^2 a}{1-3\delta}\le a\left( \frac{3\delta}{1-3\delta} +\frac{9(1+2h)}{16}\right)F
  \end{align*}
  We can treat the equality case of this inequality as the definition of $F$, effectively forcing \eqref{combined} to be true.
  Hence we can take
  \begin{align*}
    \frac{1}{F}=
    \frac{3+7\delta+6\sqrt{a(1-3\delta-a)})}{16b^2}
  \end{align*}
  But \eqref{combined} is true, therefore we proved that
  \begin{align*}
    \mu_a(a,b)\ge \frac{\pi^2\left(3+7\delta+6\sqrt{a(1-3\delta-a)}\right)}{12b^2}
  \end{align*}
  The assumption $3b^2<1-a+a^2$ in \autoref{kitesym} is equivalent to $1-3\delta-a<-4a^2+3a$. Under this assumption the above bound simplifies to
  \begin{align*}
    \mu_a(a,b)\ge \frac{\pi^2\left(3+7\delta+6a\sqrt{3-4a}\right)}{12b^2}
  \end{align*}
  
  Note that the $c$ we choose above for $3b^2=1-a+a^2$ satisfies $c=\frac{1}{4(1-a)}>\frac14$. Hence even for $a\approx 0$ (near equilateral) we are using rhombi that are not far from square, hence the Hooker-Protter bound is the most accurate according to \cite[Figure 12]{FS10} (in the notation of that paper we are dealing with rhombi with $a\ge \sqrt{3}/3$). Obviously for smaller values of $b$ we are using much smaller values of $c$, but these cases are far from critical in our proof.

  \subsection{General variational upper bound approach}\label{general}
  Variational upper bounds involving linear combinations of any number of transplanted exact eigenfunctions always have the following form
  \begin{align}\label{form}
    \mu_2(a,b)\le \frac{A(a)+B(a)b^2}{C(a)b^2},
  \end{align}
  where $A(a)$, $B(a)$ and $C(a)$ are polynomials. One way to show that $\mu_a> \mu_2$ is to first show that $C(a)>0$ for $0\le a\le 1/2$. Then prove that 
  \begin{align*}
  12A(a)+12B(a)b^2\le C(a)\pi^2(3+7a^2+7b^2-7a+6a\sqrt{3-4a}).
  \end{align*}
  This is however equivalent to
  \begin{align}\label{muamus}
  12A(a)+(12B(a)-7\pi^2C(a))b^2\le C(a)\pi^2(3+7a^2-7a+6a\sqrt{3-4a}).
  \end{align}
  Next we show that the polynomial in front of $b^2$ is positive and we get that the left hand side is increasing with $b$, while the right hand side is decreasing. Therefore we can put any upper bound for $b$ involving $a$ and we get an inequality for $a$. This inequality will have one square root, but it can be transformed into a high order polynomial inequality, and we need to prove it for all $0\le a\le1/2$.

  Note that we do not need a sharp inequality in the above inequalities, as long as we find upper bound that is not sharp for any $T(a,b)$ (except for known cases with double eigenvalue, square and equilateral).
  \subsection{Upper bound for $\mu_2(T)$ and the proof of \autoref{kitesym}}
We will take a linear combination of 3 eigenfunctions, 1 from a half-equilateral triangle and 2 from a right isosceles triangle. To be more precise we need the second eigenfunction on $T(0,1/\sqrt{3})$ and the first two nonconstant eigenfunctions from $T(1/2,1/2)$. Take
\begin{align*}
  \varphi(x,y)&=(2a-1)\cos\left( \frac{2\pi y}{3b} \right)\left(1-2\cos\left( \frac{\pi(bx-ay)}{b} \right)\right)+
\\&+4a\cos\left(\frac{\pi y}{2b} \right)\cos\left( \frac{\pi(2bx+(1-2a)y)}{2b} \right)+
  \\&+2a(2a-1)\cos\left( \frac{\pi(bx+(1-a)y)}{b} \right)\cos\left( \frac{\pi(bx-ay)}{b} \right).
  \end{align*}
  Note that only the first term is present for $a=0$, in  particular for $T(0,1/\sqrt{3})$. On the other hand only the middle term is nonzero for $T(1/2,1/2)$. Therefore we recover exact eigenfunctions for these special cases. We will not need this fact, nor that we used eigenfunction in our test function.
  We only need to know that this function integrates to $0$ over $T(a,b)$, hence it is a good test function for $\mu_2(T)$. This gives an upper bound in the form \eqref{form} with
  \begin{align*}
    C(a)&=67200\pi^2a^4+12(74976-5600\pi^2)a^3+12(15400\pi^2-182346)a^2+
    \\&\phantom{+67200\pi^2a^4}+12(72429-8400\pi^2)a+25200\pi^2
    \\&\ge
12(15400\pi^2-182346)a^2+12(72429-8400\pi^2)a+25200\pi^2,
  \end{align*}
  Coefficients for $a^2$ and $a$ are negative, hence we can put $a=1/2$ and we get $C(a)>0$ for $0\le a\le 1/2$.

  We also have
  \begin{align*}
    B(a)&=\pi^2\Big( 134400\pi^2a^4+(650880-134400\pi^2)a^3+(168000\pi^2-1867740)a^2 +
        \\&\qquad\quad+ (784800-67200\pi^2)a+(127575+16800\pi^2)\Big),\\
    A(a)&=\pi^2\Big( 134400 \pi^2 a^6+(650880-268800 \pi^2)a^5+(369600 \pi^2-2518620)a^4+
    \\&\qquad\quad+(2105752-235200 \pi^2)a^3+(89600 \pi^2-924757)a^2+
    \\&\qquad\quad+(227938-22400 \pi^2)a+(5600\pi ^2-42525)\Big).
  \end{align*}
  We need to show that the following polynomial is positive
  \begin{align*}
    \frac{12B(a)-7\pi^2C(a)}{12\pi^2}&=95200\pi^2 a^4+(126048-95200 \pi^2)a^3+(60200\pi^2-591318)a^2+
    \\&\qquad\quad+(277797-8400\pi^2)a+(2100\pi^2+127575)
  \end{align*}
  Note that only the coefficient for $a^3$ is negative, and replacing $a^3$ with $1/8$ still gives positive constant coefficient. Hence we proved that $12B-7\pi^2 C>0$.
  Therefore in \eqref{muamus} we can replace $b$ with its maximal value $\sqrt{(1-a+a^2)/3}$ and we need to prove that
  \begin{align}
    P(a)a&\le 18Q(a)a\sqrt{3-4a},\label{sqrtineq}\\
    Q(a)&=5600 \pi^2a^4+(74976-5600 \pi^2)a^3+2(7700\pi^2-91173)a^2+\nonumber\\
    &\qquad\quad+(72429-8400 \pi^2)a+2100 \pi^2,\nonumber\\
    P(a)&=380800\pi^2a^5+(504192-761600 \pi ^2)a^4+(868000 \pi^2-2869464) a^3-\nonumber\\
    &\qquad\quad-480(665 \pi^2-2682)a^2-8(44211+2450 \pi^2)a+525(347+80 \pi^2)\nonumber
  \end{align}
  Again we note that coefficients of $a^4$ and $a^3$ in $Q$ are positive, hence we can disregard these terms. While coefficients of $a^2$ and $a$ in $Q$ are negative, hence we can put $a=1/2$ and we get $Q(a)>0$.
  Therefore it is enough to show 
  \begin{align}\label{poly}
    P^2(a)-18^2Q^2(a)(3-4a)\le 0.
  \end{align}
  Note that for $T(0,1/\sqrt{3})$ and $T(1/2,1/2)$ we have exact eigenvalues. Hence we can expect that $a=0$ and $a=1/2$ are roots of the above polynomial. In fact we already eliminated $a=0$ in \eqref{sqrtineq}. However $a=1/2$ is a double root of \eqref{poly} and we can reduce the degree by 2. Numerical results suggest that the inequality is roughly true for $a\in[-0.49,0.52]$ ($(0,1/2)$ is needed). To avoid closeness of the positive root substitute $a\to1/2-a$. We still need to prove the new inequality for $a\in[0,1/2]$ and it is numerically true up to $a\approx 1$. We are left with $8$-degree polynomial inequality
  \begin{align*}
   0\ge &36252160000 \pi^4 a^8+(-95998156800 \pi^2-46412800000 \pi ^4)a^7+\\
   &+(63552393216-34277644800 \pi^2+83354880000\pi ^4)a^6+\\
   &+(-1352162962944+582294182400\pi^2-149461760000 \pi^4)a^5+\\
   &+(-3554482258800-300435206400 \pi^2+121433760000 \pi^4)a^4+\\
   &+(-1682712947520+1404232972800 \pi^2-139740160000 \pi^4)a^3+\\
   &+(4864275678312-1107844970400 \pi^2+70309120000 \pi^4)a^2+\\
   &+(-1418249685780+311172170400 \pi^2-20603520000 \pi^4)a+\\
   &+(3669120000 \pi^4-36985183200 \pi^2)
  \end{align*}
  We will reduce this polynomial to a negative constant by increasing it and simplifying at the same time. We apply the following steps
  \begin{enumerate}
    \item Coefficient for $a^8$ is positive, hence we can replace $a^8$ with $a^7/2$. Similarly for $a^6$.
    \item Coefficient for $a^4$ is positive, hence we can replace $a^4$ with $a^3\frac{\frac27+\frac72 a^2}{2}\ge a^4$. Here simple linear estimate is too rough yielding false inequality.
    \item Similarly replace $a^2$ with $a\frac{\frac12+2a^2}{2}\ge a^2$.
    \item New coefficients for $a^7$ and $a^5$ are still negative, hence we can replace $a^7$ and $a^5$ with $0$.
    \item New coefficient for $a^3$ is positive, hence we can replace $a^3$ with $a/4$. 
  \end{enumerate}
  After all these steps we get a linear function with negative coefficients, proving the inequality is true and it is strict.

  Therefore we get strict inequality $\mu_a>\mu_2$ for all cases except $T(0,1/\sqrt{3})$ and $T(1/2,1/2)$. These triangles have double eigenvalues.
  \section{Simplicity for $\mu_2$}
Note that Ba\~nuelos and Burdzy \cite{BB99} proved that if a kite is narrow enough, then $\mu_2$ of the kite is simple (Proposition 2.4(ii)) and the second eigenfunction is symmetric (Proposition 2.3). Hence the second eigenvalue is also simple for the corresponding triangle. Later, Atar and Burdzy \cite{AB04} proved simplicity for obtuse and right triangles. 

We prove simplicity of $\mu_2$ for all acute triangles (\autoref{simple}). Our approach is to show that $\mu_2+\mu_3>2\mu_2$ by finding appropriate bounds for both sides, similarly to the approach from \autoref{kitesec}. Lower bound for the left side can again be obtained using the unknown trial function method developed in \cite{LSminN,LSminD}. In particular we can use \autoref{unknown}. 

In this section we will be working with triangles $T(a,b)$ with vertices $(-1,0)$, $(1,0)$ and $(a,b)$. Acute triangles can be described using the following conditions (see \autoref{fig:split}): $a\ge 0$, $a^2+b^2\ge1$ (acute), $(a+1)^2+b^2\le 4$ (the longest side on the $x$-axis). The appropriate version of \autoref{unknown} states that
    \begin{align*}
      \mu_2(a,b)+\mu_3(a,b)\ge C_{a,b,c,d}(\mu_2(c,d)+\mu_3(c,d))
    \end{align*}
    is true if
    \begin{align*}
      ( (a-c)^2+d^2)(1-\gamma)+2b(a-c)\delta+b^2\gamma\le d^2/C_{a,b,c,d},
    \end{align*}
    where $\delta$ and $\gamma$ are some numbers (unfortunately unknown) depending only on $a$ and $b$ and satisfying $|\delta|\le 1/2$ and $0\le\gamma\le 1$.
 
\begin{figure}[t]
  \begin{center}
\begin{tikzpicture}[xscale=4,yscale=4]
  \fill (0,{sqrt(3)}) circle (0.02) node [left] {\tiny $(0,\sqrt{3})$}; 
  \fill (0,1) circle (0.02) node [below left=-2pt] {\tiny $(0,1)$};
  \fill (1,0) circle (0.02) node [below=-2pt] {\tiny $(1,0)$};
  \fill (-1,0) circle (0.02) node [below=-2pt] {\tiny $(-1,0)$};
  \fill (1/2,{sqrt(3)/2}) circle (0.02) node [below left=-2pt] {\tiny $(1/2,\sqrt{3}/2)$};
  \fill (1,{2/sqrt(3)}) circle (0.02) node [right] {\tiny $(1,2/\sqrt{3})$};
  \draw[loosely dashed] (-1,0) -- (0,{sqrt(3)}) -- (1,0) -- cycle; 
  \draw[loosely dashed] (-1,0) -- (0,1) -- (1,0); 
  \draw[loosely dashed] (-1,0) -- (1,{2/sqrt(3)}) -- (1,0); 
  \draw[<->] (0,2) |- (1.2,0);
  \draw[fill, fill opacity=0.1] (1,0) arc (0:90:1) -- (0,{sqrt(3)}) arc (60:0:2);
  \draw (0,{sqrt(3)}) node [rotate=-40,above right] {\tiny $(a+1)^2+b^2=4$}; 
  \draw[dotted] (1/2,0) -- (1/2,1.39);
  \draw[dotted] (0,1.08) node [left=-2pt] {\tiny $1.08$} -- (0.8,1.08); 
  \clip (1,0) arc (0:90:1) -- (0,{sqrt(3)}) arc (60:0:2);
  \draw[thick] (1/2,0) -- (1/2,1.5);
  \draw[thick] (1/2,1.08) -- (0.8,1.08);
\end{tikzpicture}
  \end{center}
  \caption{Possible values of $(a,b)$ and the split into nearly equilateral, nearly degenerate and a small triangular middle area. Triangles with known eigenvalues are drawn with dashed lines with dots at vertices.}
  \label{fig:split}
\end{figure}
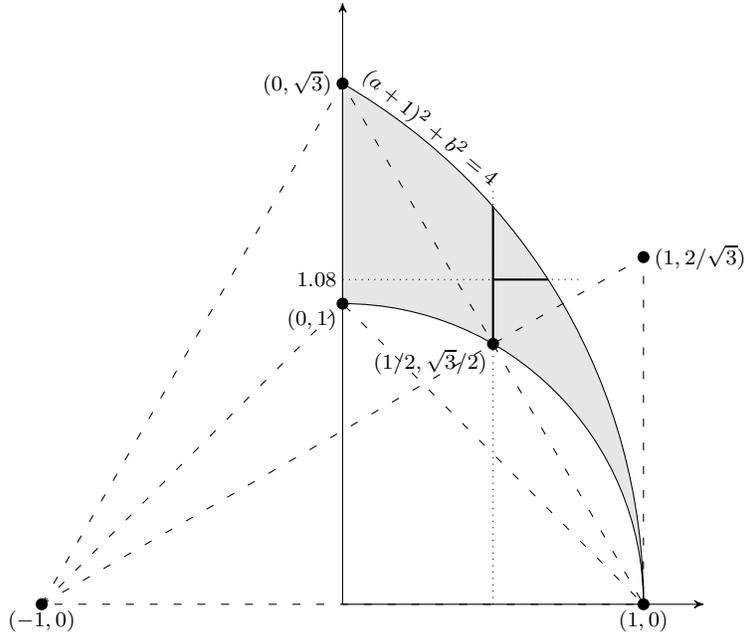

    Our goal is to prove the following chain of inequalities:
 \begin{align*}
   \mu_2(a,b)+\mu_3(a,b)\ge C(\mu_2(c,d)+\mu_3(c,d))=2U(a,b)>2\mu_2(a,b).
 \end{align*}
 We can choose any upper bound $U$ (that does not give equality for non-equilateral triangles). This leads to
 \begin{align}\label{abcdU}
   ( (a-c)^2+d^2)(1-\gamma)+2b(a-c)\delta+b^2\gamma\le d^2 \frac{\mu_2(c,d)+\mu_3(c,d)}{2U(a,b)}.
 \end{align}
 Note that we can choose at least 4 sets values of $c$ and $d$ so the eigenvalues on the right are explicit, effectively obtaining 4 inequalities involving $a, b, \gamma, \delta$. For fixed $a$ and $b$ we need to show that at least one of those inequalities is true for any admissible pair $(\gamma,\delta)$.

We will consider two upper bounds
\begin{align*}
  U_1(a,b)&=\frac{64\pi^2(a^2+b^2+3)+243(a^2+b^2-6a-3)}{288b^2},\\
  U_2(a,b)&=\frac{18}{a^2+3}
\end{align*}
The first bound can be obtained by transplanting the lowest symmetric mode of the equilateral triangle onto the triangle with vertex $(a,b)$ (compare with \autoref{general}). The second bound can be obtained by using linear function $x+a/3$ (that integrates to $0$) as a test function and it will be used for nearly degenerate triangles. The first bound gives exact eigenvalue only for equilateral triangle, the second is never accurate. Hence we can use these bounds.

Finally we choose the following pairs $(c,d)$ with explicit eigenvalues and get four special cases of \eqref{abcdU}.
\begin{itemize}
  \item $(c,d)=(0,\sqrt{3})$ (equilateral triangle)
    \begin{align}\label{ineq1}
   (a^2+3)(1-\gamma)+2ab\delta+b^2\gamma\le \frac{4\pi^2}{3U_1(a,b)}.
    \end{align}
  \item $(c,d)=(0,1)$ (right isosceles triangle)
    \begin{align}\label{ineq2}
   (a^2+1)(1-\gamma)+2ab\delta+b^2\gamma\le \frac{3\pi^2}{4U_1(a,b)}.
    \end{align}
  \item $(c,d)=(1/2,\sqrt{3}/2)$ (half of equilateral triangle - longest side on the $x$-axis)
    \begin{align}\label{ineq3}
   ( (a-1/2)^2+3/4)(1-\gamma)+(2a-1)b\delta+b^2\gamma\le\frac{2\pi^2}{3U_?(a,b)}.
    \end{align}
    Note that here we will use either $U_1$ or $U_2$, depending on the case we consider.
  \item $(c,d)=(1,2/\sqrt{3})$ (half of equilateral triangle - middle side on the $x$-axis)
    \begin{align}\label{ineq4}
   ( (a-1)^2+4/3)(1-\gamma)+2(a-1)b\delta+b^2\gamma\le\frac{8\pi^2}{9U_2(a,b)}.
    \end{align}
\end{itemize}
For example we will use the first 3 cases on nearly equilateral triangles and the last 2 on nearly degenerate. See \autoref{fig:split}.

\subsection{Nearly degenerate triangles}

We need to show that either \eqref{ineq3} or \eqref{ineq4} (both with $U_2$) is true for any $\delta$ and $\gamma$. Instead we can show that a linear combination of these inequalities with positive coefficient holds. We can choose coefficients so that $\delta$ cancels. That is $(1-a)$\eqref{ineq3}+$(a-1/2)$\eqref{ineq4}
\begin{align*}
  \frac{\gamma}{6}(3a^2+3b^2-8a+1)+\frac{1}{6}(8a-3a^2-1)-\frac{1}{81}\pi^2 (a+1)(a^2+3)\le0.
\end{align*}
Left side is linear in $\gamma$, hence it is enough to consider extremal values $\gamma=0$ and $\gamma=1$. When $\gamma=0$ we get a cubic in $a$ and we can directly check that inequality holds for $a>-6.4$. When $\gamma=1$ we get only negative coefficients for powers of $a$. Therefore we can put $a=1/2$ and the inequality is true when $|b|<1.089$. This last inequality defines the boundary of the ``nearly degenerate area''. In fact we take $1.08$ as the boundary point.

We proved that regardless of the value of $\delta$ and for any $0\le\gamma\le 1$, we get that either \eqref{ineq3} or \eqref{ineq4} is true. Hence nearly degenerate triangles have simple second eigenvalue.

\subsection{Algorithm for proving polynomial inequalities}
Other cases involve polynomial inequalities of high degree in two variables over regions. In \cite{S10} the author developed an algorithm for proving such inequalities on rectangles with lower left vertex $(0,0)$. By splitting a domain into finitely many (usually very few) rectangles one can handle more complicated domains. It is also sometimes necessary to split rectangles into smaller rectangles. In particular any equality point must be at $(0,0)$.

\lstloadlanguages{Mathematica}
\lstset{
	language=Mathematica,
	showstringspaces=false,
	breaklines,
	basicstyle={\small},
	numbers=left,
	numberstyle={\tiny},
	morekeywords={Cross,TrigReduce,AffineTransform,Del,Total,FullSimplify,RegionPlot,NotebookDirectory,Alpha},
	deletekeywords={N},
	stringstyle={\textit},
	escapeinside={*@}{@*}
}

The following Mathematica code either proves that $P(x,y)\le 0$ on $(0,dx)\times(0,dy)$ (it returns True), or fails (returns False). Note that a negative result does not mean that the inequality is false, rather that one needs to divide the rectangle into smaller rectangles. 

\begin{center}
  \begin{lstlisting}[numbers=none,captionpos=b,caption=Algorithm for proving polynomial inequalities,frame=lines]
CumFun[f_,l_]:=Rest[FoldList[f,0,l]];
PolyNeg[P_,{x_,y_},{dx_,dy_}]:=
  ((Fold[CumFun[Min[#1,0]/dy+#2&,Map[Max[#1,0]&,#1]dx+#2]&,0,
    Reverse[CoefficientList[P,{x,y}]]]//Max)<=0);
\end{lstlisting}
\end{center}

For the full description of the algorithm see Section 5 in \cite{S10}. 

\subsection{Middle area}
Here we treat triangles with $1/2\le a\le 7/10$, $1.08\le b\le \sqrt{7}/2$ and $(a+1)^2+b^2\le 4$ (small triangular area on \autoref{fig:split}). Note that here $a<b$. In this part we will only use \eqref{ineq1} and \eqref{ineq3}, both with upper bound $U_1$. Note that the coefficient for $\delta$ is positive in both cases, hence we only need to prove our inequalities for $\delta=1/2$. Therefore we need
\begin{align}
  (a^2+3)(1-\gamma)+ab+b^2\gamma\le \frac{4\pi^2}{3U_1(a,b)},\label{ineq1g}\\
  (a^2-a+1)(1-\gamma)+(a-1/2)b+b^2\gamma\le\frac{2\pi^2}{3U_1(a,b)}.\label{ineq3g}
\end{align}

The coefficient for $\gamma$ in \eqref{ineq3g} equals $-a^2+a+b^2-1\ge -a^2+a+1.08^2-1>0$. Therefore the left side is increasing with $\gamma$. Unfortunately $\gamma=1$ gives false inequality for some $(a,b)$ pairs. Take $\gamma=2/3$ to get the following inequality for $a$ and $b$ (note that $U_1>0$ as an upper bound for a positive eigenvalue)
\begin{align}\label{poly1}
 (64\pi^2(a^2+b^2+3)+243(a^2+b^2-6a-3))(a^2-a+1+3ab-3b/2+2b^2)- 2\pi^2 288b^2 \le 0.
\end{align}
Substitute $\sqrt{7}/2-b$ for $b$ and $a+1/2$ for $a$. The region of interest transforms into
\begin{align}\label{region1}
  \begin{cases}
  0\le b\le \sqrt{7}/2-1.08\le 1/4,\\
  0\le a,\\
  (a+3/2)^2+(\sqrt{7}/2-b)^2\le 4.
  \end{cases}
\end{align}
This region can be covered with 2 rectangles:
\begin{itemize}
  \item $0\le a\le 99/1000$ and $0\le b\le 1/4$: We run the algorithm with $b$ as the first variable $x$, and $a$ as the second variable $y$, on rectangle with sides $dx=1/4$ and $dy=99/1000$. Using Mathematica code from previous paragraph we run: \lstinline!PolyNeg[P,{b,a},{1/4,99/1000}].! We get that inequality \eqref{poly1} (written as $P(b,a)\le 0$) is true. Note that in this case we could also perform the steps of the algorithm by hand, since the polynomial is not overly complicated. However, in other cases we will have more complicated polynomials, for which calculations by hand would not be practical. 
  \item $99/1000\le a \le 199/1000$ and $1/9\le b\le 1/9+15/100$: We make a substitution $b\to b+1/9$ and $a\to a+99/1000$ (to put the lower left vertex at $(0,0)$) and we run the algorithm again. The inequality \eqref{poly1} is again true.
\end{itemize}
Therefore for any $|\delta|\le 1/2$ and $0\le \gamma\le 2/3$ inequality \eqref{ineq3} holds.

We will use the above procedure many times. Each time we shift a rectangle so that the lower left vertex is $(0,0)$ and we choose appropriate side lengths $dx$ and $dy$. Then we run the algorithm \lstinline!PolyNeg[P,{b,a},{dx,dy}]! to check that inequality $P(b,a)\le 0$ is true.

To finish the middle region we will show that when $2/3\le \gamma\le 1$ then either \eqref{ineq1g} or \eqref{ineq3g} is true. The claim follows if a linear combination (with positive coefficients) of the inequalities is true. Take \eqref{ineq3g}$+(b-a)$\eqref{ineq1g} and look at coefficient for $\gamma$
\begin{align*}
  b^2+a&-a^2-1+(b-a)(b^2-a^2-3)
\end{align*}
Substitute $a\to a+1/2$ and $b\to \sqrt{7}/2-b$ and run the algorithm with $dx=dy=1/4$ to check that this coefficient is nonpositive. Therefore we can take $\gamma=2/3$ in the linear combination \eqref{ineq3g}$+(b-a)$\eqref{ineq1g} and we need
\begin{align*}
  (a^2-a+1+2b^2+3ab-3/2b)+(b-a)(a^2+3+2b^2+3ab)\le\frac{2\pi^2}{U_1(a,b)}(2b-2a+1).
\end{align*}
Multiply by $U_1$, substitute $a\to a+1/2$ and $b\to \sqrt{7}/2-b$ and rearrange to the form $P(b,a)\le 0$. The result is a fifth degree polynomial in $a$ and $b$ with complicated, but explicit coefficients. This time we cover region \eqref{region1} with 3 rectangles
\begin{itemize}
  \item $0\le a\le 5/100$, $0\le b\le 1/4$,
   \item $5/100\le a\le 14/100$, $5/100\le b\le 1/4$,
   \item $14/100\le a \le 20/100$, $17/100\le b\le 1/4$.
\end{itemize}
On each rectangle (shifted so that $(0,0)$ is the lower left corner) we prove the required inequality with one application of the algorithm. Therefore for $2/3\le \gamma\le 1$ either \eqref{ineq1g} or \eqref{ineq3g} is true. Hence triangles corresponding to the middle area have simple second Neumann eigenvalue.

\subsection{Nearly equilateral triangles}
Now we treat triangles with $a\le 1/2$. In this region we need to use 3 inequalities (\eqref{ineq1},\eqref{ineq2},\eqref{ineq3}), and there is no easy way to discard $\delta$ (as before). But inequalities \eqref{ineq1}, \eqref{ineq2}, \eqref{ineq3} (all with $U_1$), are linear in both $\delta$ and $\gamma$. Hence each one is true on a halfspace given by some line. We need to make sure that the square $|\delta|\le 1/2$ and $0\le\gamma\le 1$ is fully covered by these halfspaces. 

We first show that the boundary is covered. Next we need to eliminate the possibility that the three lines (defining the halfspaces) form a triangle inside the square, and all inequalities are true outside of this triangle. All we need to do is check if one of the inequalities is true in one point of the triangle, in particular in the intersection of the other two lines.

\subsubsection{Case $\gamma=1$}
Take only the first inequality \eqref{ineq1} and put $\delta=1/2$ (has positive coefficient). We get
\begin{align*}
 3U_1(a,b)(ab+b^2)\le 4\pi^2.
\end{align*}
Hence it is enough to show that
\begin{align*}
 (64\pi^2(a^2+b^2+3)+243(a^2+b^2-6a-3))(a+b)-384b\pi^2\le 0
\end{align*}
Note that without loss of generality we can add anything to the left side, as long as the expression is nonnegative inside of the region (we make the inequality harder to prove). It can however be negative outside improving our margin of error near the boundary. Therefore we elect to prove
\begin{align*}
  (64\pi^2(a^2+b^2+3)+243(a^2+b^2-6a-3))&(a+b)-384b\pi^2+
  \\&+2000(4-(a+1)^2-b^2)(b-\sqrt{3}/2)\le 0 
\end{align*}
The added expression is negative just below the lowest point of the parameter set $(1/2,\sqrt{3}/2)$, and above the set $(a+1)^2+b^2\le 4$ (see \autoref{fig:split}). The latter property allows us to use rectangles near the equilateral triangle, since the new inequality is true on the line $b=\sqrt{3}$. This line contains equilateral triangle $a=0$, but no other admissible pair $(a,b)$. However for $b$ slightly less that $\sqrt{3}$ there is an interval $(0,a_b)$ with admissible pairs. The added expression allows us to consider only rectangular regions instead of the curved upper boundary of the parameter space on \autoref{fig:split}.

We can expect that for equilateral triangle there is equality, hence $b=\sqrt{3}$ must be mapped to the lowest corner of a rectangle. Therefore substitute $b\to\sqrt{3}-b$. Now take rectangles
\begin{itemize}
  \item $(b,a)\in [0,2/3]\times[0,1/2]$
  \item $(b,a)\in [2/3,1]\times[0,1/2]$ (here we need to shift $b$ for the algorithm)
\end{itemize}
Our algorithm proves the inequality on both rectangles, hence the case $\gamma=1$ is proved.
\subsubsection{Case $\gamma=0$}
This time we combine \eqref{ineq2} and \eqref{ineq3} so that $\delta$ disappears. That is we take $a$\eqref{ineq3}$+(1/2-a)$\eqref{ineq2}
\begin{align*}
  a(a^2-a+1)+(1/2-a)(a^2+1)\le \frac{\pi^2}{U_1(a,b)} (3(1/2-a)/4+2a/3)
\end{align*}

Now we multiply by $U_1(a,b)$, rearrange and prove the inequality using rectangles
\begin{itemize}
  \item $(b,a)\in [97/100,177/100]\times[0,1/2]$ (here we only shift $b$),
  \item $(b,a)\in [\sqrt{3/2},\sqrt{3}/2+1/5]\times[1/5,1/2]$ (here we need to shift both $a$ and $b$)
\end{itemize}
\subsubsection{$\delta=1/2$}
Here again we can work with just one inequality at the time. Start with the third one \eqref{ineq3}. As in the middle area case we note that the coefficient for $\gamma$ is positive and we use the same value $\gamma=2/3$ as in this earlier case. 
\begin{align*}
   (a^2-a+1)+3(a-1/2)b+2b^2\le\frac{2\pi^2}{U_1(a,b)}.
\end{align*}
As in the case $\gamma=1$, we can add something positive to our inequality (and negative outside the region of interest). This leads to inequality
\begin{align*}
  (64\pi^2(a^2+b^2+3)+243(a^2+b^2-6a-3))&(a^2+3ab+2b^2-a-3b/2+1)-576b^2\pi^2+
  \\&+10^4(4-(a+1)^2-b^2)(b-\sqrt{3}/2)\le 0.
\end{align*}
Again, as in the case $\gamma=1$ we expect equality for equilateral triangle, hence we substitute $b\to \sqrt{3}-b$. We prove the inequality using 3 rectangles
\begin{itemize}
  \item $(b,a)\in [0,0.14]\times[0,0.25]$
  \item $(b,a)\in [0.14,0.45]\times[0,0.5]$ (here we need to shift $b$)
  \item $(b,a)\in [0.45,0.9]\times[0,0.5]$ (we shift $b$ again).
\end{itemize}
Each time we get true inequality, hence for $\gamma\le 2/3$ inequality \eqref{ineq3} is true. For larger $\gamma$ we take just the first inequality \eqref{ineq1}. We get negative coefficient for $\gamma$ hence we choose $\gamma=2/3$. The procedure is exactly as for small $\gamma$. We rearrange and add positive term $7000(4-(a+1)^2-b^2)(b-1)^2$. Finally we use 4 rectangles to prove the inequality
\begin{itemize}
  \item $(b,a)\in [0,0.23]\times[0,0.33]$
  \item $(b,a)\in [0.23,0.66]\times[0,0.5]$ (here we need to shift $b$)
  \item $(b,a)\in [0.66,0.81]\times[0,0.5]$ (we shift $b$ again).
  \item $(b,a)\in [0.81,0.89]\times[1/3,0.5]$ (here we shift both variables)
\end{itemize}
Hence regardless of the value of $\gamma$ for $\delta=1/2$ at least one inequality holds.

\subsubsection{$\delta=-1/2$}
In the last piece of the boundary we combine the first 2 inequalities \eqref{ineq1} and \eqref{ineq2}. Since $a+b-1\ge 0$, we can take $(a+b-1)$\eqref{ineq1}$+8/7(\sqrt{3}-b)$\eqref{ineq2}. First we look at the coefficient for $\gamma$
\begin{align*}
  (a+b-1)(b^2-a^2-3)+8/7(\sqrt{3}-b)(b^2-a^2-1)
\end{align*}
This expression is a cubic in $a$ with negative coefficients for positive powers of $a$ (given that $b\le \sqrt{3}$). Therefore we can take $a=0$ to see that the coefficient is nonpositive for $1\le b\le \sqrt{3}$. However $b$ can be smaller than $1$. Nevertheless $b+a-1>0$, hence $a>1-b$. Therefore if $b<1$ then we should use $a=1-b$, and we get nonpositive coefficient again. Therefore we can put $\gamma=0$. This leads to inequality that is proved using 4 rectangles.

The first two rectangles require substitution $b\to \sqrt{3}-b$ to put equilateral triangle at the lower left corner. Then we take
\begin{itemize}
  \item $(b,a)\in [0,0.2]\times[0,0.12]$
  \item $(b,a)\in [0.2,0.55]\times[0,0.12]$ (here we need to shift $b$)
\end{itemize}

Another two rectangles require two substitutions $b\to b+\sqrt{3}/2$ and $a\to 1/2-a$, so that the half of the equilateral triangle is in the lower left corner. Now we take
\begin{itemize}
  \item $(b,a)\in [0,\sqrt{3}/2]\times[0,0.38]$
  \item $(b,a)\in [0,0.32]\times[0.38,0.5]$ (here we need to shift $b$)
\end{itemize}

Note that there is a slight overlap between the second and fourth rectangle. In fact rectangles 1, 2 and 4 form a long narrow rectangle containing all nearly isosceles cases. 
\subsubsection{Interior triangle}
Finally we check inequality \eqref{ineq3} in the triangle formed by the lines defined by inequalities \eqref{ineq1}, \eqref{ineq2} and \eqref{ineq3}. It is enough to show that inequality \eqref{ineq3} is true at the intersection of the lines given by \eqref{ineq1} and \eqref{ineq2}. Then it must be true on the whole triangle. Therefore we solve the first 2 equations and plug the solution to the last inequality in

\begin{align*}
   (a^2+3)(1-\gamma)+2ab\delta+b^2\gamma= \frac{4\pi^2}{3U_1(a,b)},\\
   (a^2+1)(1-\gamma)+2ab\delta+b^2\gamma= \frac{3\pi^2}{4U_1(a,b)},\\
   ( (a-1/2)^2+3/4)(1-\gamma)+(2a-1)b\delta+b^2\gamma\le\frac{2\pi^2}{3U_1(a,b)}.
\end{align*}

The solution is 
\begin{align*}
  \gamma&=1-\frac{7\pi^2}{24U_1},\\
  \delta&=\pi^2\frac{11+7b^2-7a^2}{48abU_1}-\frac{b}{2a}
\end{align*}
Plugging into the inequality gives
\begin{align*}
  24b^2U_1-\pi^2(11+7b^2+7a^2-4a)\le 0.
\end{align*}
The left hand side is bounded above by
\begin{align*}
  4a^2-80a+3(b^2-3)=4a(a-20)+3(b^2-3)\le 0.
\end{align*}

Hence the triangle is always covered by the halfspaces. Note that it is not clear that we need to have this paragraph, since the triangle might not even be inside of the square. However we cannot easily exclude the possibility that it actually is inside.

We proved the last remaining case, hence all nearly equilateral triangles have simple second Neumann eigenvalue. This ends the proof of \autoref{simple}.

\section*{Acknowledgements}
The work was partially supported by NCN grant 2012/07/B/ST1/03356. 

The author is grateful to Richard Laugesen for invaluable discussions on the topic of the paper, as well as suggested improvements to some arguments. 

\newcommand{\doi}[1]{%
 \href{http://dx.doi.org/#1}{doi:#1}}
\newcommand{\arxiv}[1]{%
 \href{http://front.math.ucdavis.edu/#1}{ArXiv:#1}}
\newcommand{\mref}[1]{%
\href{http://www.ams.org/mathscinet-getitem?mr=#1}{#1}}


\begin{thebibliography}{99}
\bibitem{AB04}\textbf{R.~Atar and K.~Burdzy}, \textit{On {N}eumann eigenfunctions in lip  domains},  J. Amer. Math. Soc. \textbf{17} (2004), \textit{no.~2}, 243--265  (electronic). \mref{MR2051611}

  \bibitem{BB99}
    \textbf{R.~Ba{\~n}uelos and K.~Burdzy}, \textit{On the ``hot spots'' conjecture  of {J}.\ {R}auch},  J. Funct. Anal. \textbf{164} (1999), \textit{no.~1},  1--33. \mref{MR1694534}

\bibitem{BW99}\textbf{K.~Burdzy and W.~Werner}, \textit{A counterexample to the ``hot spots''  conjecture},  Ann. of Math. (2) \textbf{149} (1999), \textit{no.~1},  309--317. \mref{MR1680567}

\bibitem{B05}\textbf{K.~Burdzy}, \textit{{The hot spots problem in planar domains with on  hole.}},  Duke Math. J. \textbf{129} (2005), \textit{no.~3}, 481--502.  \doi{10.1215/S0012-7094-05-12932-5}

\bibitem{FS10}
  \textbf{P.~Freitas and B.~Siudeja}, \textit{Bounds for the first {D}irichlet  eigenvalue of triangles and quadrilaterals},  ESAIM Control Optim. Calc. Var.  \textbf{16} (2010), \textit{no.~3}, 648--676. \mref{MR2674631}
 
\bibitem{HP61}\textbf{W.~Hooker and M.~H. Protter}, \textit{Bounds for the first eigenvalue  of a rhombic membrane},  J. Math. and Phys. \textbf{39} (1960/1961), 18--34.  \mref{MR0127610}

\bibitem{JN00}\textbf{D.~Jerison and N.~Nadirashvili}, \textit{The ``hot spots'' conjecture  for domains with two axes of symmetry},  J. Amer. Math. Soc. \textbf{13}  (2000), \textit{no.~4}, 741--772. \mref{MR1775736}

\bibitem{K85}\textbf{B.~Kawohl}, \textit{Rearrangements and convexity of level sets in  {PDE}}, Lecture Notes in Mathematics, vol. 1150, Springer-Verlag, Berlin,  1985. \mref{MR810619}

\bibitem{LSmaxN}
\textbf{R.~S. Laugesen and B.~A. Siudeja}, \textit{Maximizing {N}eumann  fundamental tones of triangles},  J. Math. Phys. \textbf{50} (2009),  \textit{no.~11}, 112903, 18. \mref{MR2567204}

\bibitem{LSminN}
\textbf{R.~S. Laugesen and B.~A. Siudeja}, \textit{Minimizing {N}eumann  fundamental tones of triangles: an optimal {P}oincar\'e inequality},  J.  Differential Equations \textbf{249} (2010), \textit{no.~1}, 118--135.  \mref{MR2644129}

\bibitem{LSminD}\textbf{R.~S. Laugesen and B.~A. Siudeja}, \textit{Dirichlet eigenvalue sums on  triangles are minimal for equilaterals},  Comm. Anal. Geom. \textbf{19}  (2011), \textit{no.~5}, 855--885. \mref{MR2886710}

\bibitem{McC02}\textbf{B.~J. McCartin}, \textit{Eigenstructure of the equilateral triangle.  {II}. {T}he {N}eumann problem},  Math. Probl. Eng. \textbf{8} (2002),  \textit{no.~6}, 517--539. \mref{MR1967479}

\bibitem{Mi12}
\textbf{Y.~Miyamoto}, \textit{A planar convex domain with many isolated ``hot  spots'' on the boundary},  Jpn. J. Ind. Appl. Math. \textbf{30} (2013),  \textit{no.~1}, 145--164. \mref{MR3022811}
   
\bibitem{P02}\textbf{M.~N. Pascu}, \textit{Scaling coupling of reflecting {B}rownian motions  and the hot spots problem},  Trans. Amer. Math. Soc. \textbf{354} (2002),  \textit{no.~11}, 4681--4702 (electronic). \mref{MR1926894}

\bibitem{R75}\textbf{J.~Rauch}, \textit{Five problems: an introduction to the qualitative  theory of partial differential equations}, Partial differential equations and  related topics ({P}rogram, {T}ulane {U}niv., {N}ew {O}rleans, {L}a., 1974),  Springer, Berlin, 1975, pp.~355--369. Lecture Notes in Math., Vol. 446.  \mref{MR0509045}

\bibitem{S10}\textbf{B.~Siudeja}, \textit{Isoperimetric inequalities for eigenvalues of  triangles},  Indiana Univ. Math. J. \textbf{59} (2010), \textit{no.~3},  1097--1120. \mref{MR2779073}

\bibitem{poly} Polymath7 project: The hot spots conjecture, \url{http://michaelnielsen.org/polymath1/index.php?title=The_hot_spots_conjecture}

\end{thebibliography}
\end{document}